\documentclass[12pt]{article}
\usepackage{graphicx}
\usepackage[latin1]{inputenc}
\usepackage{amssymb,amsfonts,amsmath,amsthm}

\def\r{\mathbb R}

\def\e{\mathbb E}
\def\h{\mathbb H}
\def\n{\mathbb N}

\newtheorem{theorem}{Theorem}[section]
\newtheorem{corollary}[theorem]{Corollary}
\newtheorem{definition}[theorem]{Definition}

\newtheorem{remark}[theorem]{Remark}

\title{Surfaces in Sol$_3$ space foliated by circles}
\author{Rafael L\'opez\\
Departamento de Geometr\'{\i}a y Topolog\'{\i}a\\
Universidad de Granada, 18071 Granada, Spain\\
\vspace*{.5cm}\\
Ana Irina Nistor\\
Faculty of Mathematics, "Al.I.Cuza'' University of Iasi\\
Bd. Carol I n. 11, 700506 Iasi, Romania}

\date{}

\begin{document}

\maketitle

\begin{abstract} In this paper we study surfaces foliated by a uniparametric
family of circles in the homogeneous space Sol$_3$. We prove that there do not exist such surfaces
with zero mean curvature or with zero Gaussian curvature. We extend this study considering  surfaces foliated by  geodesics,
equidistant lines or horocycles in totally geodesic planes and we classify all such surfaces under the assumption of minimality or flatness.
\end{abstract}

%%%%%%%%%%%%%%%%%%%%%%%%%%%%%%%%%%%%%%%%%%%%%%%%%%%%%%%%%%%%%%%
\section{Statement of the results}
%%%%%%%%%%%%%%%%%%%%%%%%%%%%%%%%%%%%%%%%%%%%%%%%%%%%%%%%%%%%%%%

In Euclidean space  $\e^3$, the plane and the catenoid are the only minimal rotational surfaces.
These surfaces are foliated by a uniparametric family of coaxial circles. We consider now a more general context
and we ask for those minimal surfaces foliated by a uniparametric family of circles in parallel planes, called  cyclic surfaces \cite{en,ni}.
A straightforward computation shows that, besides the rotational surfaces, the surface must be one of the classical family of minimal surfaces
discovered by Riemann \cite{ni,ri}. Similarly, the only cyclic surfaces in $\e^3$ with zero Gaussian curvature are rotational surfaces,
or generalized cones \cite{lo}.

The motivation of the present paper is to consider cyclic surfaces in Sol$_3$ space and to classify them under the assumptions of minimality and flatness.
This space  is a simply connected homogeneous $3$-manifold  and it belongs to one of the eight models of the geometry of Thurston \cite{th}.
As a Riemannian manifold, the space Sol$_3$  can be represented by $\r^3$ equipped with the metric
$$
\langle~,~\rangle=e^{2z}dx^2+e^{-2z}dy^2+dz^2,
$$
where $(x,y,z)$ are the canonical coordinates of $\r^3$. The space Sol$_3$ is a Lie group  with the group operation
$$
(x,y,z)\ast (x',y',z')=(x+e^{-z}x',y+e^{z}y',z+z'),
$$
and the metric $\langle~,~\rangle$ is  left-invariant. Recently, there is an activity on the study of submanifolds in Sol$_3$,
see for example \cite{dm,il,lo2,lm1,lm2,st}.

Recall that in Sol$_3$ there is not a group of rotations as in the Euclidean space $\e^3$,
that is, a uniparametric continuous group of direct isometries that leave pointwise a geodesic.
However, a rotational surface in $\e^3$ is also characterized as a surface foliated by the uniparametric circles $C_s$
all contained in planes $P_s$ whose centers form a straight line orthogonal to each $P_s$. In Sol$_3$ we have  all these concepts, such as
geodesic, totally geodesic surface, and distance. Under this viewpoint, we carry the next definition:

\begin{definition}
A circle in Sol$_3$ is a curve contained in a totally geodesic surface $P$ and equidistant from a fixed point of $P$, called the center of the circle.
\end{definition}

We point out that the distance in the circle is computed with the induced metric on $P$ by Sol$_3$,
which agrees with the one induced on the curve by the metric of the ambient space because $P$ is totally geodesic.

We recall that in  Sol$_3$ the totally geodesic surfaces are the vertical planes $P_s$ of equation $x=s$ and the planes $Q_s$ of equation $y=s$, where $s\in\r$.
Here we use the notion of 'vertical' and 'horizontal' in its usual affine concept of the set $\r^3$, where the space Sol$_3$ is modeled.

By analogy with what happens in Euclidean space, we give the next definition.

\begin{definition}
A cyclic surface in Sol$_3$ is a surface foliated by a uniparametric family of circles contained in the planes $P_s$ (resp. contained in the planes $Q_s$).
If the centers of these circles form a geodesic line orthogonal to each plane $P_s$ (resp. each plane $Q_s$), then we say that the surface is rotational.
\end{definition}

Motivated by the results of Euclidean space $\e^3$, we ask for those  cyclic surfaces in Sol$_3$ that are minimal (zero mean curvature)
or flat (zero Gaussian curvature). For example, we ask about the existence of rotational surfaces or surfaces with the same role as Riemann minimal examples.

Because the study of cyclic surfaces foliated by circles in the planes $P_s$ is analogous to the ones foliated by circles in the planes $Q_s$,
we will assume in this paper that the cyclic surfaces are always foliated by circles in the totally geodesic planes of type $P_s$.

In this sense, we obtain:

\begin{theorem}\label{t1}
 There are no  minimal cyclic surfaces in Sol$_3$.
\end{theorem}

\begin{theorem}\label{t2}
 There are no  flat cyclic surfaces in Sol$_3$.
 \end{theorem}

Both theorems will be proved in Section \ref{se-3}. On the other hand, recall that the planes $P_s$ are isometric to the hyperbolic plane $\h^2$ and the circles in
$P_s$ are curves with constant curvature $\kappa$, $\kappa>1$.  However, in the hyperbolic plane there are other curves with constant curvature $\kappa$, but
$0\leq \kappa\leq 1$. Indeed, these curves are geodesics, equidistant lines and horocycles. Therefore, we ask for minimal or flat cyclic surfaces
foliated by these curves. In this context, we obtain in Section \ref{se-4} all minimal or flat cyclic surfaces foliated by geodesics,
equidistant lines, or horocycles. See Theorems \ref{t3} and \ref{t4}.

%%%%%%%%%%%%%%%%%%%%%%%%%%%%%%%%%%%%%%%%%%%%%%%%%%%%%%%%%%%%%%%%%%
\section{Preliminaries}\label{se-2}
%%%%%%%%%%%%%%%%%%%%%%%%%%%%%%%%%%%%%%%%%%%%%%%%%%%%%%%%%%%%%%%%%%

The isometry group of Sol$_3$ has dimension $3$ and the component of the identity is generated by the following families of isometries:
\begin{eqnarray*}
& &T_{1,c}(x,y,z):=(x+c,y,z),\\
& &T_{2,c}(x,y,z):=(x,y+c,z),\\
& &T_{3,c}(x,y,z):=(e^{-c}x,e^{c}y,z+c),
\end{eqnarray*}
where $c\in\r$ is a real parameter. These isometries are left multiplications by elements of Sol$_3$ and so,
they are left-translations with respect to the structure of the Lie group. The Sol$_3$ space has  the next three foliations:
\begin{eqnarray*}
& &\mathcal{ F}_1: \{P_s=\{(s,y,z);y,z\in\r\};s\in\r\}\\
& &\mathcal{F}_2: \{Q_s=\{(x,s,z);x,z\in\r\};s\in\r\}\\
& &\mathcal{F}_3: \{R_s=\{(x,y,s);x,y\in\r\};s\in\r\}.
\end{eqnarray*}
The foliations $\mathcal{F}_1$ and $\mathcal{F}_2$  are determined by the isometry groups $\{T_{1,c};c\in\r\}$ and $\{T_{2,c};c\in\r\}$ respectively, and they
describe the (only) totally geodesic surfaces of Sol$_3$, each leaf of the foliation being isometric to a hyperbolic plane.
The foliations $\mathcal{F}_3$ are minimal surfaces, all of them being isometric to the Euclidean plane.

With respect to the metric $\langle~,~\rangle$, an orthonormal basis of left-invariant vector fields is given by
$$E_1=e^{-z}\frac{\partial}{\partial x},\hspace*{.5cm}  E_2=e^{z}\frac{\partial}{\partial y},\hspace*{.5cm}  E_3=\frac{\partial}{\partial z}.$$

The Riemannian connection ${\nabla}$ of Sol$_3$ with respect to $\{E_1,E_2,E_3\}$  is
$$\begin{array}{lll}
{\nabla}_{E_1} E_1=-E_3 & {\nabla}_{E_1}E_2=0&{\nabla}_{E_1}E_3=E_1\\
{\nabla}_{E_2} E_1=0 & {\nabla}_{E_2}E_2=E_3&{\nabla}_{E_2}E_3=-E_2\\
{\nabla}_{E_3} E_1=0 & {\nabla}_{E_3}E_2=0&{\nabla}_{E_3}E_3=0.\\
\end{array}
$$

For a surface in Sol$_3$ parameterized in local coordinates by $X=X(s,t)$ and endowed with the unit normal $N$, the formulas
of the mean curvature and Gaussian curvature are, respectively,
$$H =\frac{1}{2}\frac{En- 2Fm +Gl}{EG-F^2},\ \ K=\frac{ln-m^2}{EG-F^2},$$
with
\begin{eqnarray*}
&&E=\langle X_s,X_s\rangle,\ F=\langle X_s,X_t\rangle,\ G=\langle X_t,X_t\rangle,\\
&&l=\langle \nabla_{X_s} X_s, N\rangle,\ m=\langle \nabla_{X_s} X_t, N\rangle = \langle \nabla_{X_t} X_s, N\rangle,
\ n=\langle \nabla_{X_t} X_t, N\rangle .
\end{eqnarray*}

Each plane $P_s$ has  constant curvature $-1$ and it is isometric to  $\h^2$.
For an explicit isometry between both spaces, we consider $\h^2$ modeled by the upper half-plane model,
that is, $\h^2$ is the upper half-plane $\r^2_{+}=\{(x,y)\in\r^2;y>0\}$ endowed with the metric $ds^2=((dx)^2+(dy)^2)/y^2$.
The isometry between $P_s$ and $\h^2$ is  given by
$$
\Phi_s:P_s\rightarrow\h^2, \ \ \Phi_s(y,z)=(y,e^z).
$$
In the upper half-plane model, a circle in $\h^2$ is viewed as an Euclidean circle strictly contained  in $\r^2_+$, which it is parameterized by
$\beta(t)=(a,b)+r(\cos(t),\sin(t))$, $a,b\in\r$, $r>0$ and $b-r>0$.
The center of the circle $\beta$ is the point $c_\beta=(a,\sqrt{b^2-r^2})$ and the hyperbolic radius is $r_\beta=\log((b+r)/\sqrt{b^2-r^2})$.
Therefore the circle $\beta$ is viewed in Sol$_3$  as $C_s:=\Phi_s^{-1}(\beta)$ and its parametrization is then
\begin{equation}\label{excircle}
\alpha_s(t)=(s,a+r\cos(t),\log{(b+r\sin(t))}),\ t\in\r,
\end{equation}
see Figure \ref{fig1}. On the other hand, we consider the foliation ${\cal F}_1=\{P_s;s\in\r\}$ of the ambient space.
The geodesics of Sol$_3$ intersecting orthogonally each plane   of ${\cal F}_1$ are the horizontal lines of equation $y=a, z=b$, where $a, b\in\r$.

\begin{figure}[htbp]\begin{center}
\includegraphics[width=.5\textwidth]{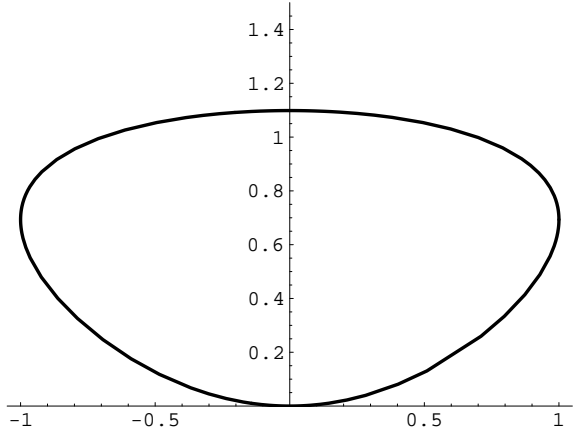}
\end{center} \caption{The circle $\Phi_s^{-1}(\beta)$ in a plane $P_s$, where $\beta(t)=(0,2)+(\cos(t),\sin(t))$}
\label{fig1}
\end{figure}

A parametrization of a rotational surface in Sol$_3$ is
$$X(s,t)=(s,a+r(s)\cos(t),\log{(b+r(s)\sin(t))}),$$
where $a,b\in\r$, $r(s)>0$ and $b-r>0$.  The line of centers is $s\longmapsto (s,a,\sqrt{b^2-r(s)^2})$.

We obtain a cyclic surface by considering the constants $a$ and $b$ as functions of the variable $s$. The
parametrization of the surface is then
\begin{equation}\label{cyclicx}
X(s,t)=(s,a(s)+r(s)\cos(t),\log{(b(s)+r(s)\sin(t))}), s\in I, t\in\r,
\end{equation}
and $I\subset \r$ is an open interval.

\begin{remark} Because our results are local, we only need to compute  $H$ and $K$ when the  variable $t$ is defined in some open interval of $\r$,
that is, when we have an arc of a circle. Thus we can consider in the parametrizations
\eqref{excircle} and \eqref{cyclicx} that $t\in J$, where $J\subset\r$ is an open interval. Then in the statements of the results we change
'foliated by circles' by 'foliated by arcs of circles'.
\end{remark}

Since in this paper we study surfaces with zero mean curvature or zero Gaussian curvature, the equations $H=0$ and $K=0$ are, respectively,  equivalent to
\begin{equation}\label{fh0}
E \langle \nabla_{X_t} X_t, \tilde{N}\rangle -  2F \langle \nabla_{X_s} X_t, \tilde{N}\rangle + G \langle \nabla_{X_s} X_s, \tilde{N}\rangle=0,
\end{equation}
and
\begin{equation}\label{fk0}
 \langle \nabla_{X_t} X_t, \tilde{N}\rangle  \langle \nabla_{X_s} X_s, \tilde{N}\rangle  -   \langle \nabla_{X_s} X_t, \tilde{N}\rangle^2=0
\end{equation}
where $\tilde{N}$ is a normal vector field to the surface but no necessarily unitary.
It will be with both equations as we will obtain our results.

The proof of our results consists in substituting parametrization \eqref{cyclicx} into \eqref{fh0} and \eqref{fk0}, obtaining a polynomial of type
\begin{equation}\label{fstrategy}
\sum_{n=0}^k (A_n(s) \cos{(nt)}+B_n(s)\sin{(nt)})=0, \ s\in I, t\in J,
\end{equation}
for some natural number $k\in\n$. For each fixed $s\in I$,
we view this expression as a polynomial in $t$; recall that $t$ belongs to an interval of $\r$.
Because the functions $\{ \cos{(nt)}, \sin{(nt)};n\in\n\}$ are linearly independent,
all coefficients $A_n(s)$ and $B_n(s)$ vanish. The explicit computations of \eqref{fh0} and \eqref{fk0} are straightforward, but complicated.
We use a symbolic program as Mathematica to check these computations.

In \eqref{fstrategy} the simplest expressions correspond to the leader coefficients $A_k$ and $B_k$. Using the equations $A_{k}=0$ and $B_k=0$,
we get information that will be used in determining the next coefficients $A_{k-1}$ and $B_{k-1}$ and we repeat the process until we find our results.

%%%%%%%%%%%%%%%%%%
\section{Proof of Theorems \ref{t1} and \ref{t2}}\label{se-3}
%%%%%%%%%%%%%%%%%%%%%%%%%%%%%%

\begin{proof}(of Theorem \ref{t1})
Using parametrization \eqref{cyclicx} of a cyclic surface in Sol$_3$ we compute
\begin{eqnarray*}
&&X_s = (b(s) + r(s)\sin(t))E_1+\frac{a'(s) + r'(s)\cos(t)}{b(s) + r(s)\sin(t)}E_2+ \frac{b'(s) + r'(s)\sin(t)}{b(s) + r(s)\sin(t)}E_3,\\
&&X_t =\frac{r(s)}{b(s) + r(s)\sin(t)}(-\sin(t) E_2+ \cos(t) E_3),
\end{eqnarray*}
and we choose the normal vector
$$
\tilde{N} =\frac{r'(s) + a'(s)\cos(t) + b'(s)\sin(t)}{(b(s) + r(s)\sin(t))^2}E_1 -\cos(t)E_2 -\sin(t)E_3.
$$
After some manipulations, the minimality condition yields an expression of type \eqref{fstrategy} with $k=5$ where
$$A_5=0,\ \ B_5=\frac{r(s)^6}{16}.$$
Because $B_5=0$, we conclude $r(s)=0$, which is a contradiction.

\end{proof}

\begin{proof}(of Theorem \ref{t2})
Following the same reasoning as in the proof of the previous theorem, the flatness condition is equivalent to an expression of type \eqref{fstrategy}
with $k=5$ again. Here
$$
A_5=0,\ \ B_5=\frac{1}{16}b(s)r(s)^5.
$$
From $B_5=0$, we have $b(s)=0$. However this is a contradiction with the fact that a circle $\Phi_s^{-1}(\beta(t))$
in $P_s$ must satisfy the condition $b-r>0$, see \eqref{excircle}.

\end{proof}

%%%%%%%%%%%%%%%%%%%%%%%%%%
\section{Surfaces foliated by geodesics, equidistant lines and horocycles} \label{se-4}
%%%%%%%%%%%%%%%%%%%%%%%%%%%%%%%%%%%%%%%%%%%%

The curvature  $\kappa$ of a circle   in Sol$_3$ satisfies $\kappa>1$ because it is isometric to a circle of $\h^2$.
In fact, if $C_s$ is parameterized as \eqref{excircle}, its curvature  is $\kappa_\beta=b/r$, computed with the orientation pointing inside.
In the hyperbolic plane  $\h^2$, the definition of a circle is equivalent to a curve with   constant curvature $\kappa$, $\kappa>1$. However,
in $\h^2$ there are other curves with constant curvature $\kappa$ with $0\leq \kappa\leq 1$. Indeed, if $\kappa=0$, the curve is a geodesic;
if $0<\kappa<1$, the curve is an equidistant line; and if $\kappa=1$, the curve is a horocycle.
In these cases, the curves are not closed such as it occurs for  a circle. In the upper half-plane model of $\h^2$,
the three types of curves are viewed as  the intersection of circles and straight-lines with $\r_+^2$. Exactly:
\begin{enumerate}
\item If $\beta$ is a geodesic, then $\beta$ is a half-circle intersecting orthogonally the line $y=0$ or it is a vertical straight-line
$\beta(t)=(a,t)$, $a\in\r$, $t>0$.
\item If $\beta$ is an equidistant line, then $\beta$ is the intersection of a circle of $\r^2$ with points above
$y=0$ with $\r^2_+$,
or $\beta$ is a straight-line of $\r^2_+$ that intersects not orthogonally the line $y=0$, which parametrizes as
$\beta(t)=(t,at+b)$, $a,b\in\r$, $a\not=0$, $t>-b/a$.
\item If $\beta$ is a horocycle, then $\beta$ is the intersection of a circle of $\r^2$ tangent to the line $y=0$  with $\r^2_+$,
or $\beta$ is a horizontal line of $\r^2_+$, that is, $\beta(t)=(t,a)$, $a>0$, $t\in\r$.
\end{enumerate}

By the isometry $\phi_s^{-1}$, we can carry all these curves to the plane $P_s$ and we call them geodesic, equidistant line and horocycle again.
Following this motivation, we generalize the concept of circle in Sol$_3$ assuming that the curvature of the curve is constant.

\begin{definition}
A cyclic surface in Sol$_3$  by geodesics (resp. equidistant lines, horocycles) is a surface foliated by a uniparametric family of geodesics
(resp. equidistant lines, horocycles)  contained in the planes $P_s$.
\end{definition}

In this section, we study cyclic surfaces in Sol$_3$ foliated by these curves and we ask if there are minimal examples or with vanishing Gaussian curvature.

If the curve $\beta$ is a part of an Euclidean circle of $\r^2_+$, by the isometry $\Phi_s^{-1}$ this curve goes to a curve in $P_s$
with the same  expression \eqref{excircle} with the difference that the condition now is $b+r\sin(t)>0$.
In fact, we have that $\alpha_s=\Phi_s^{-1}(\beta)$ is  geodesic if $b=0$, an equidistant line if $0<b<r$ and a horocycle if $0<b=r$. See Figure \ref{fig2}.
However, the expression of the parametrization of a cyclic  surface is \eqref{cyclicx}, again except the above conditions on $b$ and $r$.
Therefore, the same local computations as in Theorem \ref{t1} and \ref{t2} hold to compute the expressions of $H$ and $K$.

From the proof of Theorem \ref{t1} we obtain that there are no minimal surfaces.

Concerning the flatness, from the proof of Theorem \ref{t2}, we obtained $b(s)=0$.
In such a case, the coefficient $A_4$ is $A_4=-r(s)^2a'(s)^2/2$ and since $A_4$ must vanish, we have that $a(s)$ is a constant function $a(s)=a_0$,
$a_0\in\mathbb R$. Furthermore, $A_2=-8r(s)^2r'(s)^2$, concluding that $r(s)$ is a constant function too.

These results for surfaces parameterized by \eqref{cyclicx} will appear in the statements of Theorems \ref{t3} and \ref{t4}.

\begin{figure}[hbtp]\begin{center}
\includegraphics[width=.3\textwidth]{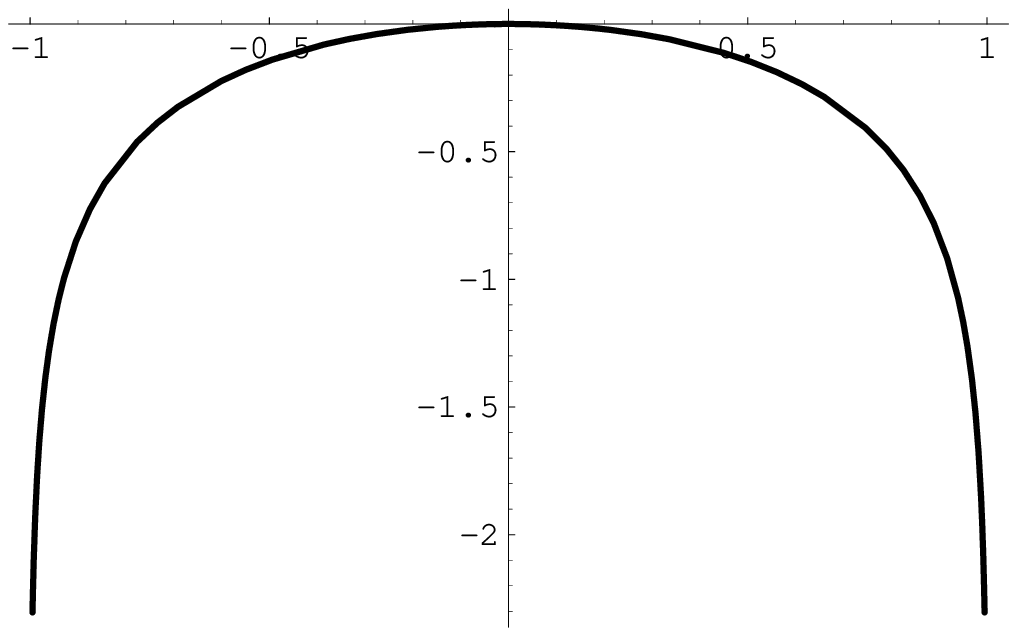}\ \includegraphics[width=.3\textwidth]{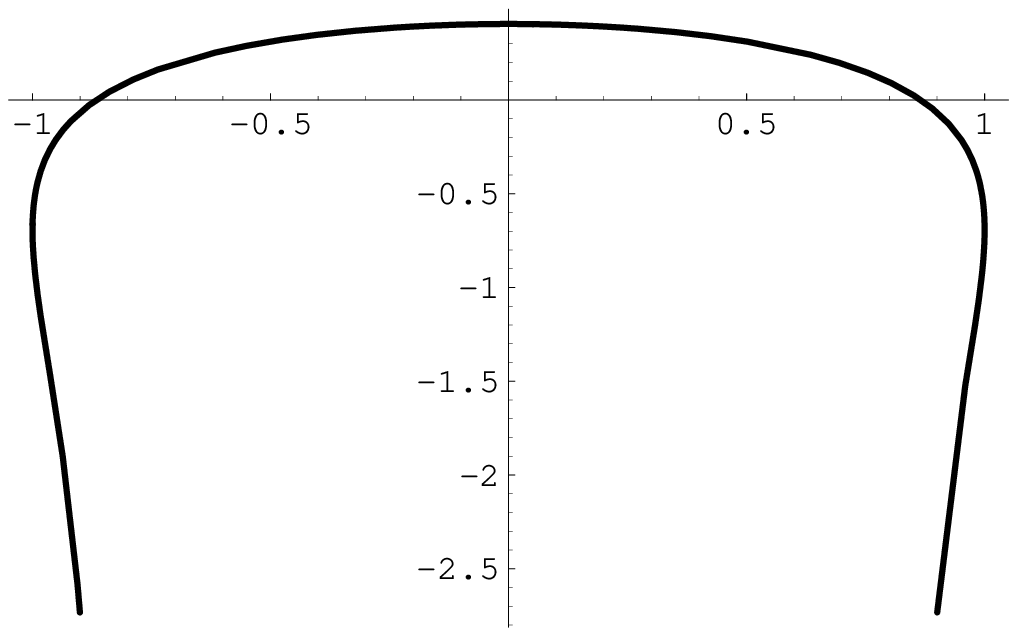}\ \includegraphics[width=.3\textwidth]{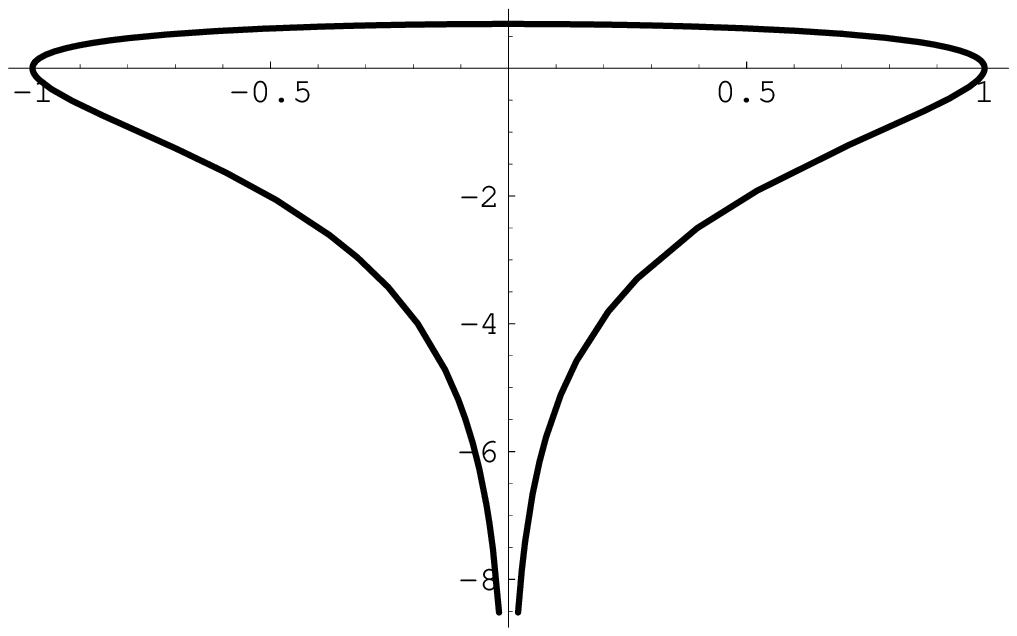}
\end{center} \caption{A geodesic, an equidistant line and a horocycle in a plane $P_s$.
The curves are obtained by $\Phi_s^{-1}(\beta)$ and $\beta(t)=(a,b)+r(\cos(t),\sin(t))$.
For a geodesic, $b=0$; for an equidistant line, $0<b<r$; and for a horocycle, $b=r>0$.}
\label{fig2}
\end{figure}

As a consequence, the  only cases that we have to study are when the curve $\beta$ is a straight-line in $\h^2$.
By the isometry $\Phi_s^{-1}$, these curves transform into:
\begin{enumerate}
\item If $\beta$ is a geodesic, $\Phi_s^{-1}(\beta(t))=(s,a,\log{(t)})$, $t>0$.
\item If $\beta$ is a equidistant line, $\Phi_s^{-1}(\beta(t))=(s,t,\log{(at+b)})$.
\item If $\beta$ is a horocycle, $\Phi_s^{-1}(\beta(t))=(s,t,\log{(a)})$, $a>0$.
\end{enumerate}

Using these parametrizations, we prove:

\begin{theorem}\label{t3}
Consider a cyclic surface in Sol$_3$ with $H=0$.
\begin{enumerate}
\item  If it is foliated by geodesics, the surface is a vertical plane $X(s,t) = (s,\lambda s+\mu,\ t)$, $\lambda,\mu\in\r$.
\item  If it is foliated by equidistant lines, the surface is $\displaystyle X(s,t)=\left(s,t,\log{\left|\frac{b_0}{s+b_1}\right|}\right)$
or  $\displaystyle X(s,t)=\left(s,t,\log{\left|\frac{a_0 t}{s+a_1}\right|}\right)$, where $a_0,a_1, b_0, b_1\in\r$.
\item  If it is foliated by horocycles, the surface is $X(s,t) =(s,t,\lambda+\log{|s+\mu|})$, $\lambda,\mu\in\mathbb{R}$.
\end{enumerate}
\end{theorem}

\begin{proof} We distinguish the three cases.

\begin{enumerate}

\item The surface is foliated by geodesics. The parametrization of the surface is given by
$X(s,t)=(s,a(s),t)$, for some smooth function $a(s)$, $t\in\r$. Now we have
$$
X_s=e^t E_1+a'(s)e^{-t}E_2,\ \ X_t=E_3,\ \ \tilde{N}=a'(s)e^{-t}E_1-e^t E_2.
$$
The minimality condition is equivalent to $a''(s)=0$.
Thus $a(s)=\lambda s+\mu$ for some constants $\lambda,\mu\in\r$. In this case, the surface is a vertical plane.

\item The surface is foliated by equidistant lines. The parametrization of the surface is
$X(s,t)=(s,t,\log{(a(s) t+b(s))})$, with $a(s) t+b(s)>0$. Here
$$
X_s=(a(s)t+b(s))E_1+\frac{a'(s)t+b'(s)}{a(s)t+b(s)}E_3,\ \ X_t=\frac{E_2+a(s)E_3}{a(s)t+b(s)},
$$
$$
\tilde{N}=-\frac{a'(s)t+b'(s)}{(a(s)t+b(s))^2}E_1-a(s)E_2+E_3.
$$
We compute the equation $H=0$ obtaining a $2$-degree polynomial in $t$ in the form $\sum\limits_{n=0}^2 A_n(s) t^n=0$.
Therefore, the coefficients $A_n(s)$ vanish on its domain.
The vanishing coefficient $A_2=(1 + a(s)^2) (-2a'(s)^2+a(s) a''(s)) $ yields $a(s)=a_0/(s+a_1)$, for some constants $a_0, a_1\in\mathbb R$.
Using this information we get
\begin{eqnarray}\displaystyle
A_1=& \frac{a_0}{(a_1+s)^4}\left(2(a_1+s)b(s)+2(2 a_1^2+a_0^2+4 a_1 s+2 s^2)b'(s)+\right.\nonumber \\
&\left.(a_1+s)(a_1^2+a_0^2+2a_1s+s^2)b''(s) \right).\nonumber
\end{eqnarray}
When $A_1$ vanishes identically, $a_0=0$ or $b(s)=0$.
The first case yields $a(s)=0$. Then, the minimality condition $H=0$ writes as $-2b'(s)^2+b(s)b''(s)=0$, with the solution $b(s)=b_0/(s+b_1)$,
$b_0, b_1\in\mathbb R$.\\
In the second case, if $b(s)=0$ then the surface is always minimal because now the expression $H=0$ is trivial.

\item The surface is foliated by horocycles. The parametrization of the surface is given by
$X(s,t)=(s,t,a(s))$, where $a(s)$ is a smooth function. We compute
$$
X_s=e^{a(s)}E_1+a'(s) E_3,\ \ X_t=e^{-a(s)}E_2,\ \ \tilde{N}=a'(s)e^{-a(s)}E_1-E_3.
$$
Here the condition $H=0$ is equivalent to $a''(s)-a'(s)^2=0$, and the solution is $a(s)=\lambda+\log{|s+\mu|}$, with $\lambda,\mu\in\r$.
\end{enumerate}
\end{proof}

\begin{remark}
Part of the surfaces obtained in Theorem \ref{t3} are invariant by the families $\{T_{1,c}\}$ or $\{T_{2,c}\}$. Indeed,
the first surface foliated by equidistant lines is $T_2$-invariant and the second one is $T_1$-invariant.
The surfaces foliated by horocycles are $T_2$-invariant. All these surfaces appeared in \cite{lm2}.
\end{remark}

In a similar manner, we   study flat cyclic surfaces foliated by geodesics, equidistant lines or horocycles.

\begin{theorem}\label{t4}
Consider a cyclic surface in Sol$_3$ with $K=0$.
\begin{enumerate}
\item If it is foliated by geodesics, then the surface is given by
\begin{itemize}
\item[(1.a)] a plane $Q_s$, $s\in\r$,
\item [(1.b)] or $X(s,t)=(s,a+r\cos(t),\log{(r\sin(t))})$, with $a,r\in\r$, $r>0$.
\end{itemize}
\item  If it is foliated by equidistant lines, then the surface is
$$X(s,t) =\left(s,t,\log{\left|\frac{1}{\sqrt{-s^2+\lambda s+\mu}}\right|}\right), \lambda,\mu\in\mathbb{R}.$$
\item   If it is foliated by horocycles, then the parametrization of the surface is
$$X(s,t) =\left(s,t,-\frac{1}{2}\log{\left|-s^2+\lambda s+\mu\right|}\right), \lambda,\mu\in\mathbb{R}.$$
\end{enumerate}
\end{theorem}
\begin{proof}
Following the same steps as in the proof of the previous theorem, we discuss the three cases.
\begin{enumerate}

\item
The condition $K=0$ for a cyclic surface foliated by geodesics is equivalent to $ a'(s)^2 = 0$, which yields $a(s)=a_0$. Thus, case (1.a) of the
first item of the theorem is proved.

Case (1.b) corresponds to a flat surface parameterized by \eqref{cyclicx}. %%%%%%% with $b=0$.

\item The flatness condition for surfaces foliated by equidistant lines is equivalent to a vanishing 4-degree polynomial, $\sum\limits_{n=0}^4 A_n(s) t^n=0$.
Thus, all the coefficients must be zero. The leader coefficient is $A_4=-a(s)^4(1+a(s)^2)$, which implies  $a(s)=0$. Now $K=0$ simply writes as
$b(s)^4+3b'(s)^2-b(s)b''(s)=0$. The solution is
$\displaystyle b(s)=\frac{\pm 1}{\sqrt{-s^2+\lambda s+\mu}}$, and it yields the parametrization from the second item of the theorem.
\item The condition $K=0$ for a cyclic surface foliated by horocycles is equivalent to $a''(s) - 2 a'(s)^2- e^{2a(s)}=0$. The solution is given by
$a(s) = -\frac{1}{2}\log{\left|-s^2+\lambda s+\mu\right|}$, $\lambda,\mu\in\mathbb{R}$, and the last item of the theorem is proved.
\end{enumerate}
\end{proof}
The surfaces foliated by equidistant lines and by horocycles obtained in Theorem \ref{t4} are now $T_2$-invariant. On the other hand, the surface foliated by geodesics is $T_1$-invariant. Again, these surfaces appeared in \cite{lm2}.

\begin{corollary}  Any minimal or flat cyclic surface in Sol$_3$ foliated by geodesics, equidistant lines or horocycles is invariant by the family of translations $\{T_{1,c}\}_{c\in\r}$ or $\{T_{2,c}\}_{c\in\r}$.
\end{corollary}

%%%%%%%%%%%%%%%%%%%%%%%%%%%%%%%%%%%%%%%%%%%%%%%%%%%%%%%%%
\section{Surfaces foliated by circles in the $z$-planes}\label{se-5}
%%%%%%%%%%%%%%%%%%%%%%%%%%%%%%%%%%%%%%%%%%%%%%%%%%%%%%%%%

The space Sol$_3$ is foliated by the surfaces $R_s$ of equation $z=s$. Each surface $R_s$ is isometric to the Euclidean plane $\e^2$
and they are minimal surfaces in Sol$_3$. We consider in $R_s$ curves equidistant from a fixed point of $R_s$. We call these curves \emph{circles} in $R_s$.
Here the distance is computed as the intrinsic distance on the surface $R_s$.
The induced metric on $R_s$ is $e^{2s}(dx)^2+e^{-2s}(dy)^2$ and an isometry $\Psi:R_s\rightarrow \e^2$ is
$\Psi(s,x,y)=(e^s x,e^{-s}y)$. Therefore, a circle in $R_s$ is the inverse image of an Euclidean circle of $\e^2$ and this circle is parametrized as
$\alpha_s(t)=(a+re^{-s}\cos(t),b+re^s\sin(t),s)$. Thus we can consider  surfaces in Sol$_3$  foliated by circles at the horizontal planes $R_s$,
which will be called cyclic surfaces. Again, we ask for those surfaces with zero mean curvature or zero Gaussian curvature.
A local parametrization of such a surface is
\begin{equation}\label{zplanes}
X(s,t)=(a(s)+r(s)e^{-s}\cos(t),b(s)+r(s)e^s\sin(t),s),\ s\in I, t\in J,
\end{equation}
where $I, J\subset \r$ are open intervals.

\begin{theorem}
There are no minimal or flat cyclic surfaces foliated by a uniparametric family of circles in the planes $R_s$.
\end{theorem}

\begin{proof}
A parametrization of such a cyclic surface is given by \eqref{zplanes}.
We have
\begin{eqnarray*}
&&X_s=(a'e^s+(r'-r)\cos(t))E_1+(b'e^{-s}+(r'+r)\sin(t))E_2+E_3,\\
&&X_t=-r\sin(t) E_1+r\cos(t)E_2.
\end{eqnarray*}
We choose the normal vector to the surface
$$\tilde{N}=\cos(t)E_1+\sin(t)E_2+(r\cos(2t)-a'e^s\cos(t)-b'e^{-s}\sin(t)-r')E_3.$$
\begin{enumerate}
\item Assume that the surface is minimal.
The equation \eqref{fh0} writes as an expression of type \eqref{fstrategy} with $k=4$. In fact, $A_4=-r(s)^3/2$ and $B_4=0$. Because $A_4=0$, then $r=0$, obtaining a contradiction.
\item If the surface is flat, the expression \eqref{fk0} is a polynomial equation \eqref{fstrategy} with $k=8$. Here $A_8=-r(s)^6/8$ and $B_8=0$. Again, we obtain  a contradiction.
\end{enumerate}
\end{proof}

{\bf Acknowledgements.}
The second author wishes to thank the Geometry Section of KU Leuven and FWO for the travel
grant K207212N to visit Universidad de Granada where the paper was initiated. This work was partially supported by the MEC-FEDER grant no. MTM2011-22547,
Junta de Andaluc\'\i a grant no. P09-FQM-5088 and CNCS-UEFISCDI (Romania) PN-II-RU-TE-2011-3-0017/2011-2014.

%%%%%%%%%%%%%%%%%%%%%%%%%%%%%%%%%%%%%%


\begin{thebibliography}{99}
%%%%%%%%%%%%%%%%%%%%%%%%%%%%%%%%

\bibitem{dm} B. Daniel, P.  Mira,
\emph{Existence and uniqueness of constant mean curvature spheres in Sol$_3$},  to appear in J. Reine Angew. Math.  DOI: 10.1515/crelle-2012-0016, 2012.

\bibitem{en} A. Enneper,
\emph{ \"{U}ber die cyclischen Fl\"{a}chen, Nach. K\"{o}nigl},  Ges. d. Wissensch. G\"{o}ttingen, Math. Phys. Kl. (1866), 243-249.



\bibitem{il} J. Inoguchi, S. Lee,
        \emph{A Weierstrass type representation for minimal surfaces in  Sol},
        Proc. Amer. Math. Soc. 146 (2008), 2209-2216.

\bibitem{lo} R. L\'opez,
\emph{ Cyclic surfaces of constant Gauss curvature},  Houston J. Math. 27 (2001), 799-805.

\bibitem{lo2} R. L\'opez,
        \emph{Constant mean curvature surfaces in Sol with non-emtpy boundary}, to appear in Houston J. Math.

\bibitem{lm1}  R. L\'opez, M.I. Munteanu,
\emph{Translation minimal surfaces in Sol$_3$}, J. Math. Soc. Japan,  64 (2012),  985-1003.

 \bibitem{lm2}  R. L\'opez, M.I. Munteanu,
\emph{ Invariant surfaces in the homogeneous  space Sol with constant curvature},  to appear in   Math. Nach.

\bibitem{ni} J.C.C. Nitsche, Lectures on Minimal Surfaces, Cambridge Univ. Press, Cambridge, 1989.

\bibitem{ri} B. Riemann, \emph{\"{U}ber die Fl\"{a}chen vom Kleinsten Inhalt be gegebener Begrenzung}, Abh. K\"{o}nigl.
Ges. d. Wissensch. G\"{o}ttingen, Mathema. Cl. 13 (1868), 329-333.

\bibitem{st} R. Souam, E.  Toubiana, \emph{Totally umbilic surfaces in homogeneous 3-manifolds},  Comm. Math. Helv. 84 (2009), 673-704.

\bibitem{th}  W. Thurston,
 Three-dimensional geometry and topology, Princenton Math. Ser. 35, Princenton Univ. Press, Princenton, NJ, (1997).



\end{thebibliography}
\end{document}